\documentclass[preprint,10pt]{elsarticle}

\usepackage{amssymb}
\usepackage{amsmath}
\usepackage{amsfonts}
\usepackage{graphicx,amsmath,amsfonts,amssymb,amscd,amsthm}
\usepackage{graphicx}
\usepackage{color}
\makeatletter
\def\EquationsBySection{\def\theequation
{\thesection.\arabic{equation}}%
\@addtoreset{equation}{section}}

\newcommand\old[1]{}
\newcommand{\pend}
{\hfill \thicklines \framebox(6.6,6.6)[l]{}}
\renewenvironment{proof}{\noindent {\it  Proof.} \rm}
{\pend}
\newtheorem{theorem}{Theorem}[section]

\newtheorem{proposition}{Proposition}[section]

\EquationsBySection \makeatother
\usepackage{indentfirst}

\journal{IJBC}

\begin{document}

\begin{frontmatter}



\title{Analytic integrability of two lopsided systems\footnote{This research is partially supported by the National Nature Science Foundation of China (11201211,11371373) and Applied
Mathematics Enhancement Program of Linyi University and the Natural Science and Engineering Research Council of Canada
(No. R2686A02)}}


\author{Feng Li$^{a,b}$, Pei Yu$^{b,}$\footnote{Corresponding author. Tel: +1 519661-2111;
E-mail address: pyu@uwo.ca},  Yirong Liu$^c$}

\address{
\small\it $^a$School of Science, Linyi University, Linyi, Shandong 276005,
P.R. China \vspace{0.0cm} \\
\small\it $^b$Department of Applied Mathematics, Western University,
London, Ontario N6A 5B7, Canada \vspace{0.0cm} \\
\small\it $^c$School of Mathematics and Statistic, Central South University,
Changsha, Hunan 410012, China}

\begin{abstract}
In this paper, we present two classes of lopsided systems and
discuss their analytic integrability. The analytic integrable conditions
are obtained by using the method of inverse integrating factor
and theory of rotated vector field. For the first class of systems,
we show that there are $n+4$ small-amplitude limit cycles enclosing
the origin of the systems for $n \ge 2$, and $10$ limit cycles for $n=1$.
For the second class of systems, we prove that there exist $n+4$
small-amplitude limit cycles around the origin of the systems for $n \ge 2$, and
$9$ limit cycles for $n=1$.
\end{abstract}


\end{frontmatter}


\section{Introduction}
Integrability is one of the most important
and difficult problems in studying ordinary differential systems.
To explain the problem, consider a planar analytic differential system, described by
\begin{equation}\label{eq1.1}
\begin{array}{ll}
\dot{u} =-v+U(u,v),\\[0.5ex]
\dot{v} =u+V(u,v),
\end{array}
\end{equation}
where dot indicates differentiation with respect to time $t$,
$U$ and $V$ are real analytic functions whose series
expansions in a neighborhood of the origin start at least from
second-order terms. By the Poincar\'{e}-Lyapunov theorem,
system \eqref{eq1.1} has a center at the origin if and only if
there exists a first integral, given in the form of
\begin{equation}\label{eq1.2}
\phi(u,v)=u^2+v^2+\sum\limits_{k+j=3}^\infty \phi_{kj}u^kv^j,
\end{equation}
where the series converges in a neighborhood of the origin.
Determining whether the origin of system (\ref{eq1.1}) is
a center or focus is called \textit{center problem}.
Another important problem in study of system \eqref{eq1.1}
is the existence of analytical first
integral in a {\it small} neighborhood of the origin of system \eqref{eq1.1}.
If there exists such an analytical first integral, the origin of
system \eqref{eq1.1} is a center, in particular, called an analytic center, see \cite{Algaba-2012}.

It is well known that it is difficult to distinguish focus from center when
the singular point is degenerate. Many research works have been done in this direction.
For example, analytic systems having a nilpotent singular point at the origin
were studied by Andreev \cite{Andreev-1958} in order to obtain
their local phase portraits. However, Andreev's results do not
distinguish focus from center. Takens \cite{Takens-1974} provided a normal form for nilpotent center of foci.
Later, Moussu \cite{Moussu-1982} found the $C^\infty$ normal form for analytic nilpotent centers. Further, Berthier and Moussu
\cite{BerthierM-1994} studied the reversibility of nilpotent centers.
Teixeria and Yang \cite{TeixeriaY-2001} analysed the relationship between reversibility and the center-focus problem,
expressed in a convenient normal form, and studied
the reversibility of certain types of polynomial vector fields.
Han {\it et al.} considered polynomial Hamiltonian
systems with a nilpotent singular point,
and they obtained necessary and sufficient conditions
for quadratic and cubic Hamiltonian systems with
a nilpotent singular point which may be a center, a cusp or a saddle, see \cite{han-2010}.
In particular, the local analytic integrability for nilpotent centers was
investigated \cite{Chavarriga-2003}, for the differential systems in the form of
\begin{equation*}
\begin{array}{l}
\dot{x}=y+P_3(x,y), \\[0.5ex]
\dot{y}=Q_3(x,y),
\end{array}\end{equation*}
which has a local analytic first integral, where $P_3$
and $Q_3$ represent homogeneous polynomials of degree three.
For third-order nilpotent singular points of a planar dynamical system,
the analytic center problem was solved by using the integrating factor method,
see for example \cite{Li-2013a}.

The Kukles system, as a well-known example, has been investigated intensively
on the existence of its limit cycles as well as its integrability.
For the following particular Kukles system,
\begin{equation*}
\begin{array}{l}
\dot{x}=y, \\[0.5ex]
\dot{y}=-x+a_1 x^2+a_2 xy+a_3 y^2+a_4 x^3+a_5 x^2y+a_6 x y^2+a_7 y^3,
\end{array}\end{equation*}
the conditions under which the origin of the system is a center
have been examined in \cite{Christopher-1990,Jin-1990,Lloyd-1990,Lloyd-1992,Rousseau-1995,Wu-1999,Zang-2008}.
More details about the Kukles system can be found in \cite{Pearson-2010}.
The so-called extended Kukles system,
\begin{equation*}
\begin{array}{l}
\dot{x}=y(1+k x), \\[0.5ex]
\dot{y}=-x+a_1 x^2+a_2 xy+a_3 y^2+a_4 x^3+a_5 x^2y+a_6 x y^2+a_7 y^3,
\end{array}\end{equation*}
has also been considered to obtain the center conditions \cite{Hill-2007a,Hill-2007b}.
Recently, center problem for some more generalized Kukles type systems have been studied \cite{Rabanal-2014,Grin-2013,Llibre-2011}. A kind of Li$\acute{e}$nard systems of type $(n,4)$ for $3 \leq n \leq 27$ was investigated and they obtained the lower bound of the maximal number of limit cycles for this kind of system in \cite{Yang-2015}.

Research on Hilbert¡¯s sixteenth problem in general usually proceeds by the investigation on specific classes of polynomial systems, much effort has been devoted in recent years to the investigation of various systems such as poincare system, Able equation, lopsided system and so on. The Kukles system is perhaps the earliest example of lopsided systems which have the following forms
\begin{equation*}
\begin{array}{l}
\dot{x}=-y, \\[0.5ex]
\dot{y}=x+P(x,y),
\end{array}\end{equation*}
or
\begin{equation*}
\begin{array}{l}
\dot{x}=-y+P(x,y), \\[0.5ex]
\dot{y}=x.
\end{array}\end{equation*}
Since then, lopsided systems have drawn more and more attention to researchers. Lopsided quartic and quintic polynomial vector fields have been studied and center conditions were obtained \cite{Salih-2002,Pons-2002}. Furthermore, Gine \cite{Gine-2002} proved that there is exactly one isochronous system for lopsided quartic system, and the origin never can be an isochronous center for lopsided quintic system. For seventh-degree lopsided system Soriano and Salih \cite{Salih-2002a} showed that the origin is a center if and only if the system is time-reversible and if
it is not, no more than seven local limit cycles can bifurcate from the origin under certain conditions. However when the origin is a degenerate singular point, there are fewer results because it is difficult to compute the Lyapunov constants. The cubic lopsided system with a nilpotent singular point has been investigated intensively. For example, Alvarez and Gasull \cite{Alvarez-2006} proved that three limit cycles can bifurcate
from a nilpotent singular point of the following system:
\begin{equation}\label{eq1.0}
\begin{array}{l}
\dot{x}=-y, \\[0.5ex]
\dot{y}=a_1 x^2+a_2 xy+a_3 y^2+a_4 x^3+a_5 x^2y+a_6 x y^2+a_7 y^3,
\end{array}\end{equation}
via an analysis based on normal forms.
Then, Liu and Li \cite{Liu-2009} showed that by making a small
perturbation to the linear terms of \eqref{eq1.0}, it can exhibit four small-amplitude limit cycles.
Bifurcation of limit cycles and center conditions for the following two
families of lopsided systems with nilpotent singularities,
\begin{equation*}
\begin{array}{l}
\dot{x}=-y+P_4(x,y), \\[0.5ex]
\dot{y}=-2x^3,
\end{array}\end{equation*}
and
\begin{equation*}
\begin{array}{l}
\dot{x}=-y+P_5(x,y), \\[0.5ex]
\dot{y}=-2x^3,
\end{array}\end{equation*}
have been considered by Li et al. \cite{Li-2013b}, where $P_4(x,y)$ and $P_5(x,y)$ represent homogeneous polynomials in $x$ and $y$ of degree four and five, respectively.
Their results show that it is more difficult to distinguish focus from center when the
singular point is degenerate. As far as analytic center of lopsided system is concerned, it is more challenging to distinguish it from focus. So, in this paper, we shall discuss analytic center conditions and bifurcation of limit cycles for two classes of lopsided systems with a cubic-order nilpotent singular point,
given by
\begin{equation}\label{eq1.3}
\begin{array}{l}
\dot{x}=y+H_3(x,y)+H_{2n+3}(x,y), \\[0.5ex]
\dot{y}=-2x^3,
\end{array}
\end{equation}
and
\begin{equation}\label{eq1.4}
\begin{array}{l}
\dot{x}=y, \\[0.5ex]
\dot{y}=-2x^3+H_3(x,y)+H_{2n+3}(x,y),
\end{array}
\end{equation}
where $H_k(x,y)$ represent a $k$th-degree homogeneous polynomial in $x$ and $y$.

The main goal of this paper is to apply the method of integrating factor
and theory of rotating vector fields to distinguish
analytic integrability conditions and to find the conditions for
analytic centers. This work is a continuation of that for the Kukles system
with a degenerate singular point.
In next section, we present some known results which are necessary
for proving the main result.
We derive the analytic center conditions for the centers of systems \eqref{eq1.3}
and \eqref{eq1.4} in Sections 3 and 4, respectively.
Finally, conclusion is drawn in Section 5.

\section{Preliminary results}

\noindent In this section, we present some relative notions and results taken from
\cite{LiuJ-2009,Liu-2010}, which will be used in the following sections.
A system whose origin is a cubic-order monodromic singular point can be written as
\begin{equation}\label{eq2.1}
\begin{array}{l}
\dot{x}=y+\mu x^2+\sum\limits_{i+2j=3}^\infty
a_{ij}x^iy^j=X(x,y),\\ [0.5ex]
\dot{y}=-2x^3+2\mu xy+\sum\limits_{i+2j=4}^\infty
b_{ij}x^iy^j=Y(x,y).
\end{array}
\end{equation}

\begin{theorem}\label{T2.1}
For any positive integer $s$ and a given number sequence
$ \, \{c_{0\beta}\}, \ \beta \ge 3$,
a formal series can be constructed successively in terms of
the coefficients $c_{\alpha\beta}$ ($\alpha\neq0$) as
\begin{equation}\label{eq2.2}
M(x, y)=y^2+\sum\limits_{\alpha+\beta=3}^\infty
c_{\alpha\beta}x^\alpha y^\beta=\sum\limits_{k=2}^\infty M_k(x, y),
\end{equation}
satisfying
\begin{equation}\label{eq2.3}
\left(\frac{\partial X}{\partial x}+\frac{\partial Y}{\partial
y}\right)M-(s+1)\left(\frac{\partial M}{\partial x}X+\frac{\partial
M}{\partial y}Y\right)=\sum\limits_{m=3}^\infty\omega_m(s, \mu)x^m,
\end{equation}
where $M_k(x, y)$ is a $k$th-degree homogeneous polynomial in
$x$ and $y$, satisfying $s\mu=0$ for all $k$.
\end{theorem}

\begin{theorem}\label{T2.2}
For $\alpha\geq1, \alpha+\beta\geq3$ in \eqref{eq2.2} and
\eqref{eq2.3}, $c_{\alpha\beta}$ can be uniquely determined by the
recursive formula,
\begin{equation}\label{eq2.4}
c_{\alpha\beta}=\frac{1}{(s+1)\alpha}(A_{\alpha-1, \beta+1}
+B_{\alpha-1, \beta+1}).
\end{equation}
For $m\geq1$, $\omega_m(s, \mu)$ can be uniquely determined by the
recursive formulae:
\begin{equation}\label{eq2.5}
\omega_m(s, \mu)=A_{m, 0}+B_{m, 0},
\end{equation}
\begin{equation}\label{eq2.6}
\lambda_m=\frac{\omega_{2m+4}(s, \mu)}{2m-4s-1}.
\end{equation}
where
\begin{equation}\begin{array}{l}
A_{\alpha\beta}=\sum\limits_{k+j=2}^{\alpha+\beta-1}[k-(s+1)(\alpha-k+1)]a_{kj}c_{\alpha-k+1, \beta-j}, \\
B_{\alpha\beta}=\sum\limits_{k+j=2}^{\alpha+\beta-1}[j-(s+1)(\beta-j+1)]b_{kj}c_{\alpha-k, \beta-j+1}.
\end{array}\end{equation}
\end{theorem}

\begin{theorem}\label{T2.3}
The origin of system \eqref{eq2.1} is an analytic center if and only if
the origin of system \eqref{eq2.1} is a center of $\infty$-class, namely,
the origin of system \eqref{eq2.1} is a center for any natural number $s$.
\end{theorem}

\section{Analytic centers of system \eqref{eq1.3}}

\noindent Now, we discuss the analytic centers of system  \eqref{eq1.3} in two cases.

\subsection{Case 1: $n=1$.}
For this case, system (\ref{eq1.3}) can be written as
\begin{equation}\label{eq3.1}
\begin{array}{ll}
\dot{x}=\!\!\!& y+a_{30}x^3+a_{21}x^2y+a_{12}xy^2+a_{03}y^3+a_{50}x^5+a_{41}x^4y\\[1.0ex]
&+a_{32}x^3y^2+a_{23}x^2y^3+a_{14}xy^4+a_{05}y^5,\\[1.0ex]
\dot{y}=\!\!\! & -2x^3.
\end{array}
\end{equation}
According to Theorem \ref{T2.1},
we can find a formal series $M(x,y)=x^4+y^2 + o((x^2+y^2)^2)$
for system \eqref{eq3.1},
such that \eqref{eq2.3} holds.
Applying the recursive formulae in Theorem \ref{T2.2} to
system \eqref{eq3.1}, with the help of Mathematica, we obtain
\begin{equation*}
\begin{split}
\omega_3&=\omega_4=\omega_5=0,\\
\omega_6&=(4s-1)a_{30},\\
\omega_7&=3(a+1)c_{03},\\
\omega_{8}&=-\frac{1}{5}(4s-3)(2 a_{12} + 5 a_{50}),\\
\omega_{9}&=0,\\
\omega_{10}&=-\frac{1}{7} (4s-5)(2 a_{32} + 3 a_{21} a_{50}) ,
\end{split}\end{equation*}
\begin{equation}\label{eq3.2}
\begin{split}
\omega_{11}&=\frac{15}{4}(s+1)c_{05}, \\
\omega_{12}&=-\frac{1}{45} ( 4 s-7)
(12 a_{14}+ 30 a_{03}a_{50}+ 5 a_{41}a_{50}) ,\\
\omega_{13}&=0,\\
\omega_{14}&= - \frac{3\,a_{50}}{77} (4 s-9)
(6 a_{23} + a_{21} a_{41} - 10 a_{50}^2) ,\\
\omega_{15}&=\frac{35}{8}(s+1)c_{07},\\
\omega_{16}&= - \frac{a_{50}}{117} (4 s-11)
(60 a_{05}+ 10 a_{03} a_{41} + a_{41}^2 - 3 a_{21} a_{50}^2),\\
\omega_{17}&=0,\\
\omega_{18}&=\frac{a_{50}}{1155} (4 s-13)
(2 a_{21}a_{41}^2 + 300 a_{03} a_{50}^2 + 9 a_{21}^2 a_{50}^2
+ 100 a_{41} a_{50}^2) ,\\
\omega_{19}&=\frac{315}{64}(s+1)c_{09},\\
\omega_{20}&=-\frac{a_{50}}{895050} ( 4 s-15)
(28 a_{21} a_{41}^4 + 252 a_{21}^2 a_{41}^2 a_{50}^2
+ 800 a_{41}^3 a_{50}^2 +567 a_{21}^3 a_{50}^4 \\
   & \hspace{1.2in} + 3600 a_{21} a_{41} a_{50}^4 + 4500 a_{50}^6) ,\\
\omega_{21}&=0,\\
\omega_{22}&=-\frac{4\,a_{50}}{235125} ( 4 s-17)
(4 a_{21}^2 a_{41}^4 + 36 a_{21}^3 a_{41}^2 a_{50}^2
+ 100 a_{21} a_{41}^3 a_{50}^2 +81 a_{21}^4 a_{50}^4\\
 & \hspace{1.2in} + 450 a_{21}^2 a_{41} a_{50}^4 - 125 a_{41}^2 a_{50}^4), \\
\omega_{23}&=\frac{693}{128}(s+1)c_{011},\\
\omega_{24}&=\frac{a_{50}}{42089726250000(s+1)}\, f_1,
\end{split}\end{equation}
where
\begin{equation*}
\begin{split}
f_1=&-15174868212 a_{21}^6 a_{41}^4 - 84454927200 a_{21}^4 a_{41}^5
+22768748000 a_{21}^2 a_{41}^6- 136573813908 a_{21}^7 a_{41}^2 a_{50}^2\\
& -1193662008000 a_{21}^5 a_{41}^3 a_{50}^2- 2087643726000 a_{21}^3 a_{41}^4 a_{50}^2 +651216400000 a_{21} a_{41}^5 a_{50}^2 \\
&- 307291081293 a_{21}^8 a_{50}^4-3661266760200 a_{21}^6 a_{41} a_{50}^4- 9313504335000 a_{21}^4 a_{41}^2 a_{50}^4\\
&+5946721200000 a_{21}^2 a_{41}^3 a_{50}^4+ 23826000000 a_{41}^4 a_{50}^4-17785962180 a_{21}^6 a_{41}^4 s\\
&- 98929404000 a_{21}^4 a_{41}^5 s+26934900000 a_{21}^2 a_{41}^6 s- 160073659620 a_{21}^7 a_{41}^2 a_{50}^2 s\\
&-1398534984000 a_{21}^5 a_{41}^3 a_{50}^2 s-2442981870000 a_{21}^3 a_{41}^4 a_{50}^2 s + 18810000000 a_{41}^4 a_{50}^4 s\\
&+ 770106000000 a_{21} a_{41}^5 a_{50}^2 s -360165734145 a_{21}^8 a_{50}^4 s+6998670000000 a_{21}^2 a_{41}^3 a_{50}^4 s\\
&- 4290086997000 a_{21}^6 a_{41} a_{50}^4 s-10903637205000 a_{21}^4 a_{41}^2 a_{50}^4 s +4967473392 a_{21}^6 a_{41}^4 s^2\\
&+ 27628675200 a_{21}^4 a_{41}^5 s^2 -7529168000 a_{21}^2 a_{41}^6 s^2+ 44707260528 a_{21}^7 a_{41}^2 a_{50}^2 s^2\\
&+390585888000 a_{21}^5 a_{41}^3 a_{50}^2 s^2+682203816000 a_{21}^3 a_{41}^4 a_{50}^2 s^2-215262400000 a_{21} a_{41}^5 a_{50}^2 s^2\\
&+ 100591336188 a_{21}^8 a_{50}^4 s^2+1198155823200 a_{21}^6 a_{41} a_{50}^4 s^2+3044973060000 a_{21}^4 a_{41}^2 a_{50}^4 s^2\\
&-1955419200000 a_{21}^2 a_{41}^3 a_{50}^4 s^2- 5016000000 a_{41}^4 a_{50}^4 s^2.
\end{split}
\end{equation*}
Based on \eqref{eq2.6} and \eqref{eq3.2}, it is easy to find the
first ten quasi-Lyapunov constants of system \eqref{eq3.1}.

\begin{theorem}\label{T3.1}
The first ten quasi-Lyapunov constants at the origin of system \eqref{eq3.1} are given by
\begin{equation}\label{eq3.3}
\begin{split}
 \lambda_1&=a_{30},\\
 \lambda_2&=\frac{1}{5}(2 a_{12} + 5 a_{50}),\\
 \lambda_3&=\frac{1}{7}(2 a_{32} + 3 a_{21} a_{50}),\\
 \lambda_4&=\frac{1}{45}(12 a_{14}+ 30 a_{03}a_{50}+ 5 a_{41}a_{50}),\\
 \lambda_5&=\frac{3\,a_{50}}{77} (6 a_{23} + a_{21} a_{41} - 10 a_{50}^2),\\
 \lambda_6&=-\frac{a_{50}}{117}(60 a_{05}+ 10 a_{03} a_{41}
+ a_{41}^2 - 3 a_{21} a_{50}^2),\\
 \lambda_7&=-\frac{a_{50}}{1155}(2 a_{21}a_{41}^2
+ 300 a_{03} a_{50}^2 + 9 a_{21}^2 a_{50}^2 + 100 a_{41} a_{50}^2),\\
 \lambda_8&=-\frac{a_{50}}{895050}(28 a_{21} a_{41}^4
+ 252 a_{21}^2 a_{41}^2 a_{50}^2 + 800 a_{41}^3 a_{50}^2
+567 a_{21}^3 a_{50}^4\\
&~~~ + 3600 a_{21} a_{41} a_{50}^4 + 4500 a_{50}^6),\\
 \lambda_9&=-\frac{4\,a_{50}}{235125}(4 a_{21}^2 a_{41}^4
+ 36 a_{21}^3 a_{41}^2 a_{50}^2 + 100 a_{21} a_{41}^3 a_{50}^2
+81 a_{21}^4 a_{50}^4\\
&~~~+ 450 a_{21}^2 a_{41} a_{50}^4 - 125 a_{41}^2 a_{50}^4),\\
\lambda_{10}&=-\frac{a_{50}}{42089726250000 ( s+1)(4s-19)}\,f_1,
  \end{split}
\end{equation}
where $\lambda_{k-1}=0$ for $k=2,\cdots,10$ have been used in
the computation of $\lambda_k$.
\end{theorem}

It follows from Theorem \ref{T3.1} that the following assertion holds.
\begin{proposition}\label{P3.1}
For  $n=1$, the origin of system \eqref{eq3.1} is an analytic center if and only if
the following conditions are satisfied:
\begin{equation}\label{eq3.4}
a_{30}=a_{12}=a_{32}=a_{14}=a_{50}=0.
\end{equation}
\end{proposition}

\begin{proof}
By setting $\lambda_1=\lambda_2=\cdots=\lambda_{10} =0$,
it is easy to get the conditions in \eqref{eq3.4}.
Assume $a_{50}\neq 0$, and denote
\begin{equation}\label{eq3.5}
\begin{split}
f_2=&28 a_{21} a_{41}^4 + 252 a_{21}^2 a_{41}^2 a_{50}^2+ 800 a_{41}^3 a_{50}^2 +567 a_{21}^3 a_{50}^4+ 3600 a_{21} a_{41} a_{50}^4 + 4500 a_{50}^6,\\
f_3=&4 a_{21}^2 a_{41}^4 + 36 a_{21}^3 a_{41}^2 a_{50}^2+ 100 a_{21} a_{41}^3 a_{50}^2 +81 a_{21}^4 a_{50}^4+ 450 a_{21}^2 a_{41} a_{50}^4 - 125 a_{41}^2 a_{50}^4.
\end{split}
\end{equation}
Then, we have
\begin{equation*}
\begin{array}{l}
\begin{split}
R_1= & \ {\rm Resultant} [f_2,f_3,a_{21}]\\
= & \ 252226880859375 a_{50}^{28} (37 a_{41}^6 + 36000 a_{41}^3 a_{50}^4
+ 864000 a_{50}^8),\\
R_2= & \ {\rm Resultant} [f_2,f_1,a_{21}]\\
=& \ -750785873641864353168750000000000000000 a_{50}^{44}
(-879390304066912a_{41}^{12} \\
&+ 47983547106994035360 a_{41}^9 a_{50}^4
+49445533255803715842660 a_{41}^6 a_{50}^8\\
& +1456057532744172471928500 a_{41}^3 a_{50}^{12}
+6498810664995012399669375 a_{50}^{16}\\
& - 2759277767198304 a_{41}^{12} s +169297706825726316960 a_{41}^9 a_{50}^4 s\\
&+173470743593716632941700 a_{41}^6 a_{50}^8 s+5108374765584631369552500 a_{41}^3 a_{50}^{12} s\\
& +22851124455468570581840625 a_{50}^{16} s-2080601609429376 a_{41}^{12} s^2\\
&+151804543373289707520 a_{41}^9 a_{50}^4 s^2+154422423262500291638820 a_{41}^6 a_{50}^8 s^2\\
&+4547533910484627929569500 a_{41}^3 a_{50}^{12} s^2+20400917084512169654885625 a_{50}^{16} s^2 \\
& +610343850576768 a_{41}^{12} s^3 -33200851039501326720 a_{41}^9 a_{50}^4 s^3\\
&-34214947244661526011900 a_{41}^6 a_{50}^8 s^3-1007536381850376947992500 a_{41}^3 a_{50}^{12} s^3\\
&-4496729306788344912103125 a_{50}^{16} s^3+573878672057856 a_{41}^{12} s^4\\
&-49876076993258065920 a_{41}^9 a_{50}^4 s^4-50424620011202182011120 a_{41}^6 a_{50}^8 s^4\\
& -1484973484158264557562000 a_{41}^3 a_{50}^{12} s^4
\!-\!6678213700042142946607500 a_{50}^{16} s^4\\
& -217264690435584 a_{41}^{12} s^5 +18276082133674805760 a_{41}^9 a_{50}^4 s^5\\
&+18496906029642804955200 a_{41}^6 a_{50}^8 s^5+544720023506226322440000 a_{41}^3 a_{50}^{12} s^5\\
&+2448661869754899507450000 a_{50}^{16} s^5+ 19914634381312 a_{41}^{12} s^6\\
&-1701986627481384960 a_{41}^9 a_{50}^4 s^6-1721646399680525295360 a_{41}^6 a_{50}^8 s^6\\
& -50701299015707667936000 a_{41}^3 a_{50}^{12} s^6
+227963727065799050760000 a_{50}^{16} s^6).
\end{split}
\end{array}
\end{equation*}
With the aid of Mathematica, we obtain for $\forall s\in Z^+$
\begin{equation*}
\begin{split}
G_1= & \ {\rm Resultant} [R_1,R_2,a_{41}]\\
= & \ -182848672642886912449902102931881129741668701171875a_{50}^{96}
(1 +s)^6(-19 + 4 s)^6\\
& \ \times (12242160594943288477497249258950767957\\
&+57187190996418911124243473597501985540 s\\
& +84210057837841105190444817587559944702 s^2\\
&+22053341878592957414426973876225026580 s^3\\
&-34746447450361087057581921863631440523 s^4\\
&-7190180552428847800895138514692327280 s^5\\
&+8952012886140489676856041653019558112 s^6\\
&-1982180847477328550724618213150339840 s^7\\
&+138354459536790840295491820367594752 s^8)^3\neq 0.
\end{split}
\end{equation*}
So there are no solutions for the set of equations, $f_1=f_2=f_3=0$, implying
that there do not exist other analytic center conditions for system \eqref{eq3.1}
if $a_{50} \ne 0$.

Under the conditions in \eqref{eq3.4}, system \eqref{eq3.1} becomes
\begin{equation}\label{eq3.6}
\begin{array}{l}
\begin{split}
\dot{x}&=y+a_{21}x^2y+a_{03}y^3+a_{41}x^4y+a_{23}x^2y^3+a_{05}y^5,\\
\dot{y}&=-2x^3.
\end{split}
\end{array}\end{equation}
Obviously, system \eqref{eq3.6} is symmetric with the $y$-axis.
According to Theorem 11 in \cite{Li-2013a}, the origin is an analytic
center of system \eqref{eq3.1}.
\end{proof}

Proposition \ref{P3.1} implies that
\begin{theorem}\label{T3.2}
The necessary and sufficient conditions for the origin of system \eqref{eq3.1}
being an analytic center are determined from vanishing of the first ten quasi-Lyapunov
constants, that is, the conditions given in Proposition \ref{P3.1} are satisfied.
\end{theorem}

When the cubic-order nilpotent singular point, $O(0, 0)$
is a $10$th-order weak focus, it is easy to show that the
perturbed system of \eqref{eq3.1}, given by
\begin{equation}\label{eq3.7}
\begin{array}{l}
\begin{split}
\dot{x}=& \ \delta x+ y+a_{30}x^3+a_{21}x^2y+a_{12}xy^2+a_{03}y^3+a_{50}x^5+a_{41}x^4y\\
&+a_{32}x^3y^2+a_{23}x^2y^3+a_{14}xy^4+a_{05}y^5,\\
\dot{y}=& \  \delta y-2x^3,
\end{split}
\end{array}\end{equation}
can generate ten limit cycles
enclosing an elementary node at the origin of system \eqref{eq3.9}.

Theorem 2.2 in \cite{Liu-2010} implies the following result,
\begin{theorem}\label{T3.3}
If the origin of system \eqref{eq3.7} is a $10$th-order weak focus,
then within a small neighborhood of the origin, for $0<\delta\ll1$,
perturbing the coefficients of system \eqref{eq3.7} can yield
ten small-amplitude limit cycles bifurcating from the elementary node $O(0,0)$.
\end{theorem}

\begin{proof}
The origin of system \eqref{eq3.7} is a $10$th-order weak focus if and only if
\begin{equation*}
\begin{split}
a_{30}&= 0; a_{12} = -\frac{5 a_{50}}{2}; a_{32} = -\frac{3 a_{21} a_{50}}{2};\\
a_{14}&=-\frac{5}{12}(6 a_{03} a_{50} + a_{41} a_{50}); \\
a_{23}&=\frac{1}{6}(-a_{21} a_{41} + 10 a_{50}^2);\\
a_{05}&=\frac{1}{60} (-10 a_{03} a_{41} - a_{41}^2 + 3 a_{21} a_{50}^2);\\
a_{03}&=\frac{1}{(300 a_{50}^2)}(-2 a_{21} a_{41}^2 -9 a_{21}^2 a_{50}^2 - 100 a_{41} a_{50}^2).
\end{split}
\end{equation*}
and
\begin{equation*}
\begin{split}
J_0&=\frac{\partial(\lambda_1,\lambda_2,\lambda_3,\lambda_4,\lambda_5,\lambda_6,\lambda_7,\lambda_8,\lambda_9)}{\partial(a_{30},a_{12},a_{32},a_{14},a_{23},a_{05},a_{03},a_{21},a_{41})}\\
&=-\frac{2048 a_{50}^3}{56772622027996875}(224 a_{21}^2 a_{41}^7 + 3024 a_{21}^3 a_{41}^5 a_{50}^2 + 11000 a_{21} a_{41}^6 a_{50}^2\\
&~~+13608 a_{21}^4 a_{41}^3 a_{50}^4 + 108900 a_{21}^2 a_{41}^4 a_{50}^4 +123500 a_{41}^5 a_{50}^4\\
&~~+ 20412 a_{21}^5 a_{41} a_{50}^6+311850 a_{21}^3 a_{41}^2 a_{50}^6 + 783000 a_{21} a_{41}^3 a_{50}^6\\
&~~+200475 a_{21}^4 a_{50}^8 + 1022625 a_{21}^2 a_{41} a_{50}^8 + 450000 a_{41}^2 a_{50}^8).
\end{split}
\end{equation*}
Furthermore,
\begin{equation*}
\begin{array}{l}
\begin{split}
R_5=& \ {\rm Resultant} [f_2,\frac{J_0}{a_{50}^3},a_{21}]\\
= & \ 97384 a_{41}^9 - 98391600 a_{41}^6 a_{50}^4 + 24582976875 a_{41}^3 a_{50}^8 +
 80858250000 a_{50}^{12},\\
& \ {\rm Resultant} [R_5,R_1,a_{21}]\\
=& \ 1928337060674939567063811915624524516160347438902935330714052761390000\\
&\ 00000000000000000000000 a_{50}^{72}\neq 0.
\end{split}
\end{array}
\end{equation*}
So Theorem 2.2 in \cite{Liu-2010} yields the conclusion holds.
\end{proof}

\subsection{Case 2: $n\geq2$.}

\noindent For this case, system \eqref{eq1.3} can be written as
\begin{equation}\label{eq3.8}
\begin{array}{l}
\begin{split}
\dot{x}=& \ y+x(a_{30}x^3+a_{21}x^2y+a_{12}xy^2+a_{03}y^3
+a_{2n+3,0}x^{2n+3}\\
&+a_{2n+2,1}x^{2n+2}y +a_{2n+1,2}x^{2n+1}y^2+\cdots+a_{1,2n+2}xy^{2n+2}\\
&+a_{0,2n+3}y^{2n+3})\equiv  \ X_1(x,y),\\
\dot{y}= & \ -2x^3.
\end{split}
\end{array}\end{equation}

\begin{theorem}\label{T3.4}
For $n\geq2$, the origin of system \eqref{eq3.8} is at most a $(n\!+\!4)$th-order
weak focus. If the origin of system \eqref{eq3.8} is a $(n\!+\!4)$th-order weak focus,
then within a small neighborhood of the origin,
perturbing the coefficients of system \eqref{eq3.8} can yield
$n+4$ small-amplitude limit cycles around the elementary node $O(0, 0)$.
\end{theorem}

\begin{proof}
For a nilpotent system, in order to study the dynamical behavior
in the neighborhood of the origin, we could consider $y$ and $x^2$
to be infinitesimal equivalence in the neighborhood of the origin, see \cite{Liu-2010}.
Construct a comparison system,
\begin{equation}\label{eq3.9}
\begin{array}{l}
\begin{split}
\dot{x}&=y+x(a_{21}x^2y+a_{03}y^3+a_{2n+2,1}x^{2n+2}y+\cdots
+a_{1,2n+2}xy^{2n+2})\\
& \equiv X_2(x,y),\\
\dot{y}&=-2x^3,
\end{split}
\end{array}\end{equation}
which shows that the system is symmetric with the $x$-axis, and so the
origin $O(0,0)$ is a center.

Next, we compute the determinant of system \eqref{eq3.7} to obtain
$$
\begin{array}{rl}
J_1= \!\!\! & \det \left[\begin{array}{cccc}
X_1(x,y)&-2x^3\\
X_2(x,y)&-2x^3
\end{array}\right] \\ [3.0ex]
= \!\!\! & -2x^4(a_{30}x^2+a_{12}y^2+a_{2n+3,0}x^{2n+2}\\
&+a_{2n+1,2}x^{2n}y^2+\cdots+a_{3,2n}x^2y^{2n}+a_{1,2n+2}y^{2n+2}).
\end{array}
$$
By treating the $y$ and $x^2$ as infinitesimal equivalence in the neighborhood of the origin,
we have
\begin{equation}\begin{split}
J_1&=-2x^4(a_{30}x^2+a_{12}x^4+a_{2n+3,0}x^{2n+2}
+a_{2n+1,2}x^{2n+4}\\
&+\cdots+a_{3,2n}x^{4n+2}+a_{1,2n+2}x^{4n+4}),
\end{split}\end{equation}
which implies that $a_{30},\, a_{12},\, a_{2n+3,0},\, a_{2n+1,2},\, \cdots,\,
a_{3,2n},\, a_{1,2n+2}$ could be taken as the focus values of system
\eqref{eq3.7}. So for $n \ge 2$,
the origin of system \eqref{eq3.8} is at most an $(n\!+\!4)$th-order weak focus.
According to Theorem 4.1.5 in \cite{LiuJ-2009},
within a small neighborhood of the origin,
perturbing the coefficients of system \eqref{eq3.8} can yield
$n+4$ small-amplitude limit cycles around the elementary node $O(0, 0)$.
\end{proof}

Furthermore, similar to Proposition \ref{P3.1}, we have the following result.
\begin{theorem}\label{T3.5}
For $n\geq2$,
the origin of system \eqref{eq3.8} is an analytic center if and only if
\begin{equation}
a_{30}=a_{12}=a_{2n+3,0}=a_{2n+1,2}=\cdots=a_{3,2n}=a_{1,2n+2}=0.
\end{equation}
\end{theorem}

\begin{proof}
When $a_{30}=a_{12}=a_{2n+3,0}=a_{2n+1,2}=\cdots=a_{3,2n}=a_{1,2n+2}=0,$ system \eqref{eq3.8} could be rewritten as
\begin{equation}\label{eq3.16}
\begin{array}{l}
\begin{split}
\dot{x}&=y+a_{21}x^2y+a_{03}y^3+a_{41}x^4y+a_{23}x^2y^3+a_{05}y^5\\
&+\cdots+a_{2n+2,1}x^{2n+2}y+\cdots+a_{0,2n+3}y^{2n+3},\\
\dot{y}&=-2x^3.
\end{split}
\end{array}\end{equation}
Obviously, system \eqref{eq3.16} is symmetric with the $y$-axis.
According to Theorem 11 in \cite{Li-2013a}, the origin is an analytic
center of system \eqref{eq3.8}.
\end{proof}

\section{Analytic centers of system \eqref{eq1.4}}

\noindent Now we turn to discuss the analytic center conditions for system \eqref{eq1.4}.
It also has two cases.

\subsection{Case A: $n=1$.}

For this case, system \eqref{eq1.4} can be written as
\begin{equation}\label{eq4.1}
\begin{array}{l}
\begin{split}
\dot{x}&=y,\\
\dot{y}&=-2x^3+a_{21}x^2y+a_{12}xy^2+a_{03}y^3+a_{50}x^5+a_{41}x^4y\\
&~~~+a_{32}x^3y^2+a_{23}x^2y^3+a_{14}xy^4+a_{05}y^5,\\
\end{split}
\end{array}
\end{equation}
for which we can find a formal series $M(x,y)=x^4+y^2
+ o((x^2+y^2)^2)$
according to Theorem \ref{T2.1},
provided that \eqref{eq2.3} holds. Carrying out calculations with help of
Mathematica and applying the recursive formulae in Theorem \ref{T2.2} to
system \eqref{eq4.1}, we obtain
\begin{equation*}
\begin{split}
\omega_3&=\omega_4=\omega_5=0,\\
\omega_6&=-\frac{1}{3}(4s-1)a_{21},\\
\omega_7&=3(s+1)c_{03},\\
\omega_{8}&=-\frac{1}{5}(4s-3)(6 a_{03} + a_{41}),\\
\omega_{9}&=0,\\
\omega_{10}&=-\frac{1}{7} (4s-5) (2 a_{03}a_{12}-2a_{23}+3a_{03}a_{50}) ,\\
\omega_{11}&=\frac{15}{4}(s+1)c_{05},\\
\omega_{12}&=\frac{1}{30} (4s-5)
(40 a_{05}- 4 a_{03} a_{32} - 2 a_{03} a_{12} a_{50} - 5 a_{03} a_{50}^2),\\
\omega_{13}&=0,\\
\omega_{14}&=\frac{a_{03}}{154} ( 4 s-9)(48 a_{03}^2 - 40 a_{14}
+ 12 a_{12} a_{32} + 6 a_{12}^2 a_{50}\\
& \hspace{0.5in} + 12 a_{32} a_{50} +21 a_{12} a_{50}^2 + 18 a_{50}^3) .
\end{split}
\end{equation*}
Then, for $a_{12}+2a_{50}\neq 0$,
\begin{equation}\label{4.2}\begin{split}
\omega_{15}&=\frac{35}{8}(s+1)c_{07},\\
\omega_{16}&=\frac{a_{03}}{520} (4 s-11)
(64 a_{03}^2 a_{12}+ 16 a_{32}^2 + 128 a_{03}^2 a_{50} + 16 a_{12} a_{32} a_{50}\\
 & \hspace{0.5in} +4 a_{12}^2 a_{50}^2 + 32 a_{32} a_{50}^2 + 20 a_{12} a_{50}^3 + 23 a_{50}^4) ,\\
\omega_{17}&=0,\\
\omega_{18}&=-\frac{a_{03}}{61600 (a_{12} + 2 a_{50})} (4 s-13)
(4a_{32}+ 2 a_{12}a_{50}+ 5a_{50}^2)(112a_{12}^2a_{32}- 432a_{32}^2\\
&\hspace{0.5in} +56a_{12}^3a_{50}-96a_{12}a_{32} a_{50}+200a_{12}^2a_{50}^2 - 640a_{32}a_{50}^2\\
&\hspace{0.5in}+ 120a_{12}a_{50}^3 - 85a_{50}^4) ,\\
\omega_{19}&=\frac{315}{64}(1+s)c_{09},\\
\omega_{20}&=-\frac{a_{03}}{40840800(a_{12} + 2 a_{50})^2}
(4 s-15) (4a_{32}+ 2 a_{12}a_{50}+ 5a_{50}^2)(14372996 a_{12}^4 a_{32}\\
& \hspace{0.5in} - 63894256 a_{12}^2 a_{32}^2+ 34076160 a_{32}^3+7186498 a_{12}^5 a_{50}- 10734116 a_{12}^3 a_{32} a_{50}\\
&\hspace{0.5in} -12772032 a_{12} a_{32}^2 a_{50}+ 28572751 a_{12}^4 a_{50}^2-99036264 a_{12}^2 a_{32} a_{50}^2\\
& \hspace{0.5in} + 45544768 a_{32}^2 a_{50}^2 +26958196 a_{12}^3 a_{50}^3- 39087216 a_{12} a_{32} a_{50}^3),\\
\omega_{21}&=0,
\end{split}\end{equation}
\begin{equation*}
\begin{split}
\omega_{22}&=\frac{a_{03}}{11639628000(a_{12} + 2 a_{50})^2}(4a_{32}
+ 2 a_{12}a_{50}+ 5a_{50}^2)\,f_4;
\end{split}\end{equation*}
and for $a_{12}+2a_{50}= 0$,
\begin{equation*}
\begin{split}
\omega_{16}&=\frac{a_{03}}{520} (4 s-11) (-4 a_{32} + a_{50}^2)
(4 a_{32} + a_{50}^2) ,\\
\omega_{17}&=0,
\end{split}\end{equation*}
and in addition if $a_{32}=\frac{a_{50}^2}{4}$,
\begin{equation*}\begin{split}
\omega_{18}&=\frac{9\,a_{03}}{7700} (4 s-13) (4 a_{32} + a_{50}^2)
(24 a_{03}^2 + a_{50}^3) ,\\
\omega_{19}&=\frac{315}{64}(s+1)c_{09},\\
\omega_{20}&=-\frac{7\,a_{03}}{13260} (4 s-15) a_{50}^6 ,\\
\omega_{21}&=0 ,\\
\omega_{22}&=-\frac{a_{03}a_{50}^7}{13856700 (1 + s)}(5391 - 205861 s + 66718 s^2) ;
\end{split}
\end{equation*}
if $a_{32}=-\frac{a_{50}^2}{4}$,
\begin{equation*}\begin{split}
\omega_{18}&=0,\\
\omega_{19}&=\frac{315}{64}(s+1)c_{09},\\
\omega_{20}&=\frac{2\,a_{03}}{5525}(4 s-15)(16a_{03}^2 +a_{50}^3) (27a_{03}^2 + 2a_{50}^3) ,\\
\omega_{21}&=0 ,\\
\omega_{22}&=-\frac{4\,a_{03}a_{50}}{1154725 (1 + s)}
(16a_{03}^2 +a_{50}^3) (27a_{03}^2 + 2a_{50}^3) ,\\
\end{split}\end{equation*}
where
\begin{equation*}
\begin{split}
f_4=& -175151884140096 a_{12}^5 a_{32} + 900479057104608 a_{12}^3 a_{32}^2
-870691837997952 a_{12} a_{32}^3\\
&  - 87575942070048 a_{12}^6 a_{50}+191710117401504 a_{12}^4 a_{32} a_{50}
+ 108696440458488 a_{12}^2 a_{32}^2 a_{50} \\
&-135343772601984 a_{32}^3 a_{50} - 348204560750520 a_{12}^5 a_{50}^2
+1449441463187484 a_{12}^3 a_{32} a_{50}^2\\
& -1227387674864544 a_{12} a_{32}^2 a_{50}^2
- 328545834076764 a_{12}^4 a_{50}^3+705733010157654 a_{12}^2 a_{32} a_{50}^3\\
&  - 180965555611680 a_{32}^2 a_{50}^3 -196530256579516 a_{12}^5 a_{32} s
+ 956545497558256 a_{12}^3 a_{32}^2 s\\
& -775708121404800 a_{12} a_{32}^3 s - 98265128289758 a_{12}^6 a_{50} s
+188198718433828 a_{12}^4 a_{32} a_{50} s \\
& - 92031891970176 a_{32}^3 a_{50} s -390699835897045 a_{12}^5 a_{50}^2 s
+ 1519172249991516 a_{12}^3 a_{32} a_{50}^2 s\\
&-1080068393470816 a_{12} a_{32}^2 a_{50}^2 s - 368634374118786 a_{12}^4 a_{50}^3 s +690504693129726 a_{12}^2 a_{32} a_{50}^3 s\\
& - 123060065895200 a_{32}^2 a_{50}^3 s +80500800862640 a_{12}^5 a_{32} s^2
- 396831061224512 a_{12}^3 a_{32}^2 s^2 \\
&+336523546570752 a_{12} a_{32}^3 s^2 + 40250400431320 a_{12}^6 a_{50} s^2
-79597335692936 a_{12}^4 a_{32} a_{50} s^2 \\
&-56511998888512 a_{12}^2 a_{32}^2 a_{50} s^2
+ 43311880631808 a_{32}^3 a_{50} s^2 +160035098537960 a_{12}^5 a_{50}^2 s^2\\
& -632261525950008 a_{12}^3 a_{32} a_{50}^2 s^2
+470151299998208 a_{12} a_{32}^2 a_{50}^2 s^2
+150997488382038 a_{12}^4 a_{50}^3 s^2
\end{split}\end{equation*}
\begin{equation*}
\begin{split}
& +142715605813496 a_{12}^2 a_{32}^2 a_{50} s
-292287689531688 a_{12}^2 a_{32} a_{50}^3 s^2
+ 57905489716480 a_{32}^2 a_{50}^3 s^2.
\end{split}
\end{equation*}
Based on \eqref{eq2.6} and \eqref{4.2}, it is easy to find the
first nine quasi-Lyapunov constants of system \eqref{eq4.1}.

\begin{theorem}\label{T4.1}
The first nine quasi-Lyapunov constants evaluated at origin of system \eqref{eq4.1} are given by
\begin{equation*}
\begin{split}
\lambda_1&=-\frac{1}{3}a_{21},\\
\lambda_2&=-\frac{1}{5}(6 a_{03} + a_{41}),\\
\lambda_3&=-\frac{1}{7}(2 a_{03}a_{12}-2a_{23}+3a_{03}a_{50}) ,\\
\lambda_4&=\frac{1}{30}(40 a_{05}- 4 a_{03} a_{32} - 2 a_{03} a_{12} a_{50} - 5 a_{03} a_{50}^2),\\
\lambda_5&=\frac{a_{031}}{154} (48 a_{03}^2 - 40 a_{14}+ 12 a_{12} a_{32}
+ 6 a_{12}^2 a_{50} + 12 a_{32} a_{50} + 21 a_{12} a_{50}^2 + 18 a_{50}^3).
\end{split}
\end{equation*}
Then, for $a_{12}+2a_{50}\neq 0$,
\begin{equation*}\begin{split}
\lambda_6&=\frac{a_{03}}{520}(64 a_{03}^2 a_{12}
+ 16 a_{32}^2 + 128 a_{03}^2 a_{50} + 16 a_{12} a_{32} a_{50} +
 4 a_{12}^2 a_{50}^2 + 32 a_{32} a_{50}^2\\
 &~~~~+ 20 a_{12} a_{50}^3 + 23 a_{50}^4) ,\\
\lambda_7&=-\frac{a_{03}}{61600 (a_{12} + 2 a_{50})}
(4a_{32}+ 2 a_{12}a_{50}+ 5a_{50}^2)(112a_{12}^2a_{32}- 432a_{32}^2+56a_{12}^3a_{50}\\
& -96a_{12}a_{32} a_{50}+200a_{12}^2a_{50}^2 - 640a_{32}a_{50}^2
+ 120a_{12}a_{50}^3 - 85a_{50}^4) ,\\
\lambda_8&=-\frac{a_{03}}{40840800(a_{12} + 2 a_{50})^2}(4a_{32}
+ 2 a_{12}a_{50}+ 5a_{50}^2)(14372996 a_{12}^4 a_{32}\\
& \ \ \ - 63894256 a_{12}^2 a_{32}^2+ 34076160 a_{32}^3
+7186498 a_{12}^5 a_{50} - 10734116 a_{12}^3 a_{32} a_{50}\\
& \ \ \ -12772032 a_{12} a_{32}^2 a_{50} + 28572751 a_{12}^4 a_{50}^2
-99036264 a_{12}^2 a_{32} a_{50}^2\\
& \ \ \ + 45544768 a_{32}^2 a_{50}^2 +26958196 a_{12}^3 a_{50}^3
- 39087216 a_{12} a_{32} a_{50}^3) ,\\
\lambda_9&=\frac{a_{03}}{11639628000(a_{12} + 2 a_{50})^2}(4a_{32}
+ 2 a_{12}a_{50}+ 5a_{50}^2)\,f_4;
\end{split}
\end{equation*}
while for $a_{12}+2a_{50}= 0$,
\begin{equation*}
\begin{split}
\lambda_6&=\frac{a_{03}}{520} (-4 a_{32} + a_{50}^2)
(4 a_{32} + a_{50}^2) ,\\
\end{split}\end{equation*}
and in addition if $a_{32}=\frac{a_{50}^2}{4}$,
\begin{equation*}\begin{split}
\lambda_7&=\frac{9\,a_{03}}{7700}(4 a_{32} + a_{50}^2)
(24 a_{03}^2 + a_{50}^3) ,\\
\lambda_8&=-\frac{7\,a_{03}}{13260} a_{50}^6 ,\\
\lambda_9&=-\frac{a_{03}a_{50}^7}{13856700 (1 + s)}(5391 - 205861 s + 66718 s^2) ;
\end{split}
\end{equation*}
if $a_{32}=-\frac{a_{50}^2}{4}$,
\begin{equation*}\begin{split}
\lambda_7&=0,\\
\lambda_8&=\frac{2\,a_{03}}{5525}(16a_{03}^2 +a_{50}^3) (27a_{03}^2 + 2a_{50}^3) ,\\
\lambda_9&=-\frac{4\,a_{03}a_{50}}{1154725 (1 + s)}
(16a_{03}^2 +a_{50}^3) (27a_{03}^2 + 2a_{50}^3) ,\\
\end{split}\end{equation*}
where $\lambda_{k-1}=0$ for $k=2,\cdots,9$ have been used in
computing $\lambda_k$.
\end{theorem}

Furthermore, the following result can be easily obtained.
\begin{proposition}\label{P4.1}
For $n=1$, origin of system \eqref{eq4.1} is an analytic center
if and only if one of the following conditions holds:
\begin{equation}\label{eq4.3}
a_{21}=a_{03}=a_{41}=a_{23}=a_{05}=0;
\end{equation}
\begin{equation}\label{eq4.4}\begin{split}
&a_{21}=a_{14}=a_{05}=0, \quad a_{41}=-6a_{03}, \quad\\
&a_{23}= \textstyle\frac{1}{2}(2 a_{12} + 3 a_{50})a_{03},\quad
a_{03}^2=-\frac{1}{16}a_{50}^3;
\end{split}\end{equation}
\begin{equation}\begin{split}\label{eq4.5}
&a_{21}\!=\!a_{05}\!=\!0, \ a_{41}\!=\!-6a_{03}, \ a_{12}\!=\!-2a_{50}, \
a_{23}\!=\!- \textstyle\frac{1}{2}a_{03}a_{50}, \\
&\ a_{32}\!=\!\frac{1}{4}a_{50}^2, \
a_{03}^2\!=\!-\frac{2}{27}a_{50}^3, \ a_{14}\!=\!-\frac{1}{72}a_{50}^3.
\end{split}\end{equation}
\end{proposition}

\begin{proof}
It is easy to get the conditions \eqref{eq4.3}-\eqref{eq4.5} by setting
$\lambda_1=\lambda_2=\cdots=\lambda_{9}=0$.
When $a_{50}\neq 0$, let
\begin{equation*}
\begin{split}
f_5=&\ 112 a_{12}^2 a_{32} - 432 a_{32}^2 + 56 a_{12}^3 a_{50}
- 96 a_{12} a_{32} a_{50} +200 a_{12}^2 a_{50}^2 - 640 a_{32} a_{50}^2\\
&+ 120 a_{12} a_{50}^3 - 85 a_{50}^4,\\
 f_6= & \ 14372996 a_{12}^4 a_{32} - 63894256 a_{12}^2 a_{32}^2
+ 34076160 a_{32}^3 +7186498 a_{12}^5 a_{50}\\
&- 10734116 a_{12}^3 a_{32} a_{50}-12772032 a_{12} a_{32}^2 a_{50} + 28572751 a_{12}^4 a_{50}^2\\
&-99036264 a_{12}^2 a_{32} a_{50}^2 + 45544768 a_{32}^2 a_{50}^2 \\
&+26958196 a_{12}^3 a_{50}^3 - 39087216 a_{12} a_{32} a_{50}^3.
\end{split}
\end{equation*}
Then, we obtain
\begin{equation*}
\begin{split}
R_3= & \ {\rm Resultant}[f_4,f_5,a_{32}]\\
=& -200206652313600 a_{50}^3 (a_{12} + 2 a_{50})^4 (72 a_{12}^5
+ 652 a_{12}^4 a_{50} +2694 a_{12}^3 a_{50}^2 \\
&+ 6043 a_{12}^2 a_{50}^3 + 7092 a_{12} a_{50}^4 + 3463 a_{50}^5),\\
R_4= & \ {\rm Resultant} [f_4,f_6,a_{32}]\\
= & -289298612593152000 a_{50}^3 (a_{12} + 2 a_{50})^4 (278734084992 a_{12}^7+3588989330016 a_{12}^6 a_{50}\\
&+ 21026493958464 a_{12}^5 a_{50}^2+71854303647672 a_{12}^4 a_{50}^3+ 152324731255716 a_{12}^3 a_{50}^4 \\
&+197669760539760 a_{12}^2 a_{50}^5 + 144211704399495 a_{12} a_{50}^6+45502270176438 a_{50}^7 \\
&+1651136278704 a_{12}^7s+18838630495800 a_{12}^6 a_{50} s+ 98153819816532 a_{12}^5 a_{50}^2 s\\
&+294846821571666 a_{12}^4 a_{50}^3 s + 534120897148782 a_{12}^3 a_{50}^4 s+565346904516561 a_{12}^2 a_{50}^5 s\\
&-310952885702031 a_{12} a_{50}^6 s+62534145621954 a_{50}^7 s + 1859553268056 a_{12}^7 s^2
\end{split}
\end{equation*}
\begin{equation*}
\begin{split}
&+17577034105268 a_{12}^6 a_{50} s^2 + 76054378966082 a_{12}^5 a_{50}^2 s^2+181916843290105 a_{12}^4 a_{50}^3 s^2\\
&+238094684972342 a_{12}^3 a_{50}^4 s^2+147267056136244 a_{12}^2 a_{50}^5 s^2+ 19507558179230 a_{12} a_{50}^6 s^2 \\
&-6987410240074 a_{50}^7 s^2- 1714957345728 a_{12}^7 s^3-17270801713568 a_{12}^6 a_{50} s^3\\
&- 80278393721648 a_{12}^5 a_{50}^2 s^3-212296048411048 a_{12}^4 a_{50}^3 s^3-328607615124608 a_{12}^3 a_{50}^4 s^3\\
& -285216632755972 a_{12}^2 a_{50}^5 s^3- 121914622605938 a_{12} a_{50}^6 s^3-19568176640258 a_{50}^7 s^3\\
&+ 301352777088 a_{12}^7 s^4+3072037299008 a_{12}^6 a_{50} s^4+ 14465726297408 a_{12}^5 a_{50}^2 s^4 \\
&+ 38896515312256 a_{12}^4 a_{50}^3 s^4+ 61684793716376 a_{12}^3 a_{50}^4 s^4+55602984609100 a_{12}^2 a_{50}^5 s^4\\
&+ 25319000517368 a_{12} a_{50}^6 s^4+4451109045332 a_{50}^7 s^4).
\end{split}
\end{equation*}
Further, with the aid of Mathematica, we obtain for $\forall s \in Z^+$
\begin{equation*}
\begin{split}
G_2= & \ {\rm Resultant} [R_3,R_4,a_{12}]\\
= & \ 30984189289342953910272000 a_{50}^{35} (1 + s)^5 (-17 +4 s)^5\\
&\times(-123287750793562256929839075859953216\\
& -33902812452688795016920021129342044624 s\\
&-180855325034978657368019444342423080236 s^2\\
& -1067066959204615961659004741488392865575 s^3\\
&-3328343437962444375340762099992891472110 s^4\\
& -4773196954655562390848005555854946921459 s^5\\
&+14241540803784759916727236436410335714320 s^6 \\
& -10732088467496096467063502795815305475120 s^7\\
&+3721436248399857364295558363131840668032 s^8 \\
& -625676462230475741935920059840273863168 s^9\\
&+41454785818979861302571809000901003264 s^{10})\neq 0.\\
\end{split}
\end{equation*}
The above calculations indicate that the equations
$f_4=f_5=f_6=0$ do not have real solutions, namely, there do not exist
other analytic center conditions for system \eqref{eq4.1} if $a_{50}\ne 0$.

When the conditions in \eqref{eq4.3} hold, system \eqref{eq4.1} becomes
\begin{equation}\label{eq4.6}
\begin{array}{l}
\begin{split}
\dot{x}&=y,\\
\dot{y}&=-2x^3+a_{12}xy^2+a_{50}x^5+a_{32}x^3y^2+a_{14}xy^4.\\
\end{split}
\end{array}\end{equation}
Obviously, this system is symmetric with the $y$-axis, implying that
the origin of system \eqref{eq4.6} is an analytic center due to
Theorem 11 in \cite{Li-2013a}.

When the conditions in \eqref{eq4.4} are satisfied, system \eqref{eq4.1} becomes
\begin{equation}\label{eq4.7}
\begin{array}{l}
\begin{split}
\dot{x}=& \ y,\\
\dot{y}= & \ \frac{1}{4}(-8 x^3 + 4a_{50} x^5 - 24a_{03} x^4 y
+ 4 a_{12} x y^2 - 2 a_{12} a_{50} x^3 y^2 \\
&\quad \, -5 a_{50}^2 x^3 y^2 + 4 a_{03} y^3 + 4 a_{03} a_{12} x^2 y^3
+ 6 a_{03} a_{50} x^2 y^3).
\end{split}
\end{array}
\end{equation}
Introducing the transformation,
\begin{equation*}
x=x, \quad
y=\frac{(-2 + a_{50} x^2) z}{2 (-2 +a_{50}x^2 + a_{03} x z)},
\end{equation*}
and time scaling,
$$
T=\frac{2 (-2 + a_{50} x^2)^3t}{-2 + a_{50} x^2 - 2 a_{03} x y},
$$
into system \eqref{eq4.7} results in
\begin{equation}\label{eq4.8}
\begin{array}{l}
\begin{split}
\frac{dx}{dT} =& \ z(a_{50}x^2-2)^2,\\
\frac{dz}{dT} =&-\frac{1}{4} x \big( 128 x^2 - 192 a_{50}x^4 + 96a_{50}^2 x^6
- 16a_{50}^3 x^8 - 16 a_{12} z^2 \\
&+24  a_{12} a_{50} x^2 z^2+ 20 a_{50}^2 x^2 z^2 - 96 a_{03}^2 x^4 z^2 -12  a_{12} a_{50}^2 x^4 z^2\\
&- 20 a_{50}^3 x^4 z^2 + 48 a_{03}^2 a_{50} x^6 z^2+2  a_{12} a_{50}^3 x^6 z^2 + 5 a_{50}^4 x^6 z^2\\
& + 32 a_{03}^3 x^5 z^3+2 a_{03} a_{50}^3 x^5 z^3),
\end{split}
\end{array}
\end{equation}
which is symmetric with the $z$-axis because $a_{03}^2=-\frac{a_{50}^3}{16}$.
Thus, according to Theorem 11 in \cite{Li-2013a},
the origin of system \eqref{eq4.7} is an analytic center.

Similarly, when the conditions in \eqref{eq4.5} hold, system \eqref{eq4.1} becomes
\begin{equation}\label{eq4.9}
\begin{array}{l}
\begin{split}
\dot{x}= & \ y,\\
\dot{y}= &\ \frac{1}{72}(-144 x^3 + 72 a_{50} x^5 - 432 a_{03} x^4 y
- 144a_{50} x y^2 -18 a_{50}^2 x^3 y^2\\
&+ 72 a_{03} y^3- 36 a_{03} a_{50} x^2 y^3 -a_{50}^3 x y^4),
\end{split}
\end{array}\end{equation}
for which there exists an analytic integrating factor,
$$
u(x,y)=\frac{e^{-\frac{3}{8} a_{50}^2 x^4}}
{(1 - \frac{1}{2}a_{50} x^2 + \frac{3}{4} a_{50} x y)^4},
$$
indicating that the origin of system \eqref{eq4.9} is an analytic center.
\end{proof}

Therefore, Proposition \ref{P4.1} implies the following result.
\begin{theorem}\label{T4.1}
The necessary and sufficient conditions for the origin of system \eqref{eq4.1}
being an analytic center are given by the vanishing of the first nine quasi-Lyapunov constants,
that is, one of the conditions in Proposition \ref{P4.1} is satisfied.
\end{theorem}

Similarly, when the cubic-order nilpotent singular point $O(0, 0)$
is a $9$th-order weak focus, it is easy to prove that the
perturbed system of \eqref{eq4.1}, given by
\begin{equation}\label{eq4.10}
\begin{array}{l}
\begin{split}
\dot{x}=&\ \delta x+y,\\
\dot{y}=&\ \delta y-2x^3+a_{21}x^2y+a_{12}xy^2+a_{03}y^3+a_{50}x^5+a_{41}x^4y\\
&+a_{32}x^3y^2+a_{23}x^2y^3+a_{14}xy^4+a_{05}y^5,\\
\end{split}
\end{array}\end{equation}
can generate nine limit cycles
enclosing an elementary node at the origin.
The proof is similar to that for Theorem 6.

\begin{theorem}\label{T4.2}
If the origin of system \eqref{eq4.10} is a $9$th-order weak focus,
then within a small neighborhood of the origin, for $0<\delta\ll1$,
system \eqref{eq4.10} can yield
nine small-amplitude limit cycles around the elementary node $O(0,0)$.
\end{theorem}

\begin{proof}
The origin of system \eqref{eq4.10} is a $9$th-order weak focus if and only if
\begin{equation*}
\begin{split}
a_{21}&= 0; \\
a_{41}&= -6a_{03}; \\
a_{23}&= \frac{1}{2}a_{03}(2 a_{12} + 3 a_{50});\\
a_{05}&=\frac{1}{40}(4 a_{03} a_{32} + 2 a_{03} a_{12} a_{50} + 5 a_{03} a_{50}^2); \\
a_{14}&=\frac{3}{40}(16 a_{03}^2 + 4 a_{12} a_{32} + 2 a_{12}^2 a_{50} + 4 a_{32} a_{50} + 7 a_{12} a_{50}^2 + 6 a_{50}^3);\\
a_{03}^2&=\frac{1}{64 (a_{12} + 2 a_{50})} (-16 a_{32}^2 - 16 a_{12} a_{32} a_{50} - 4 a_{12}^2 a_{50}^2 - 32 a_{32} a_{50}^2 -
 20 a_{12} a_{50}^3 - 23 a_{50}^4).
\end{split}
\end{equation*}
and
\begin{equation*}
\begin{split}
J_2&=\frac{\partial(\lambda_1,\lambda_2,\lambda_3,\lambda_4,\lambda_5,\lambda_6,\lambda_7,\lambda_8)}{\partial(a_{21},a_{41},a_{23},a_{05},a_{14},a_{03},a_{32},a_{12})}\\
&=-\frac{(4 a_{32} + 2 a_{12} a_{50} + 5 a_{50}^2)}{38846808000 (a_{12} + 2 a_{50})^3}(16 a_{32}^2 + 16 a_{12} a_{32} a_{50} + 4 a_{12}^2 a_{50}^2 + 32 a_{32} a_{50}^2+20 a_{12} a_{50}^3 + 23 a_{50}^4)^2\\
&\times (9658653312 a_{12}^6 a_{32}^2 -94084036480 a_{12}^4 a_{32}^3 + 233154747648 a_{12}^2 a_{32}^4\\
& -58883604480 a_{32}^5 + 9658653312 a_{12}^7 a_{32} a_{50} -65797494336 a_{12}^5 a_{32}^2 a_{50} + 14579296576 a_{12}^3 a_{32}^3 a_{50}\\
& +288547550208 a_{12} a_{32}^4 a_{50} + 2414663328 a_{12}^8 a_{50}^2 +28912166304 a_{12}^6 a_{32} a_{50}^2\\
& - 406316235904 a_{12}^4 a_{32}^2 a_{50}^2 +936674204416 a_{12}^2 a_{32}^3 a_{50}^2 - 127467869184 a_{32}^4 a_{50}^2\\
& +19144952176 a_{12}^7 a_{50}^3 - 109760863872 a_{12}^5 a_{32} a_{50}^3 -244778217136 a_{12}^3 a_{32}^2 a_{50}^3 \\
&+ 1036308069632 a_{12} a_{32}^3 a_{50}^3 +49040513184 a_{12}^6 a_{50}^4 - 522094417624 a_{12}^4 a_{32} a_{50}^4\\
& +985367488400 a_{12}^2 a_{32}^2 a_{50}^4 - 6731934720 a_{32}^3 a_{50}^4 +18296999196 a_{12}^5 a_{50}^5 \\
&- 545453749756 a_{12}^3 a_{32} a_{50}^5 +1124651917440 a_{12} a_{32}^2 a_{50}^5 - 114440868400 a_{12}^4 a_{50}^6\\
& +131607099120 a_{12}^2 a_{32} a_{50}^6 + 107989619520 a_{32}^2 a_{50}^6 -173404183515 a_{12}^3 a_{50}^7 \\
&+ 352746916240 a_{12} a_{32} a_{50}^7 -62864473110 a_{12}^2 a_{50}^8 + 39921858880 a_{32} a_{50}^8 +12556768140 a_{12} a_{50}^9).
\end{split}
\end{equation*}
Furthermore,
\begin{equation*}
\begin{array}{l}
\begin{split}
R_6=& \ {\rm Resultant} [f_5,J_2,a_{32}]\\
= & \ (49 a_{12}^3 + 196 a_{12}^2 a_{50} + 357 a_{12} a_{50}^2 +466 a_{50}^3)^2\times(225792 a_{12}^9 + 3214400 a_{12}^8 a_{50}\\
& +21345968 a_{12}^7 a_{50}^2 + 87227792 a_{12}^6 a_{50}^3 +246151320 a_{12}^5 a_{50}^4 + 509498952 a_{12}^4 a_{50}^5\\
& +785491407 a_{12}^3 a_{50}^6 + 864775852 a_{12}^2 a_{50}^7 +601885864 a_{12} a_{50}^8 + 196001002 a_{50}^9),\\
& \ {\rm Resultant} [R_6,R_3,a_{12}]\\
=& \ -110020282897692638814502001588125819626479649518678944188134258621907\\
&66734966784000000 a_{50}^{15}\neq 0.
\end{split}
\end{array}
\end{equation*}
So Theorem 2.2 in \cite{Liu-2010} yields the conclusion holds.
\end{proof}

\subsection{Case B: $n\geq2$.}

\noindent For this case, system \eqref{eq1.4} can be written as
\begin{equation}\label{eq4.11}
\begin{array}{l}
\begin{split}
\dot{x}=& \ y,\\
\dot{y}=& -2x^3+(a_{21}x^2y+a_{12}xy^2+a_{03}y^3+a_{2n+3,0}x^{2n+3}\\
&+a_{2n+2,1}x^{2n+2}y+a_{2n+1,2}x^{2n+1}y^2+\cdots+a_{1,2n+2}xy^{2n+2}\\
&+a_{0,2n+3}y^{2n+3})\\
\equiv & \ Y_1(x,y).
\end{split}
\end{array}
\end{equation}

\begin{theorem}\label{T4.3}
For $n\geq2$, the origin of system \eqref{eq4.11} is at most a $(n\!+\!4)$th-order
weak focus. If the origin of system \eqref{eq4.11} is a $(n\!+\!4)$th-order
weak focus,  then within a small neighborhood of the origin of its perturbed system,,
perturbing the coefficients of system \eqref{eq4.11} can yield
$n+4$ small-amplitude limit cycles enclosing the elementary node $O(0, 0)$.
\end{theorem}

\begin{proof}
The proof is similar to that for Theorem \ref{T3.4}. We construct a comparison system for  system \eqref{eq4.1},
\begin{equation}
\label{eq4.12}
\begin{array}{l}
\begin{split}
\dot{x}&=y,\\
\dot{y}&=-2x^3+x(a_{12}xy^2+a_{2n+3,0}x^{2n+3}+\cdots+a_{2,2n+1}x^2y^{2n+1})\\
&\equiv Y_2(x,y).
\end{split}
\end{array}\end{equation}
It is easy to see that system \eqref{eq4.12} is symmetric with the $x$-axis,
and so $O(0,0)$ is a center.

Next, we compute the determinant of system \eqref{eq4.12}, yielding
$$
\begin{array}{rl}
J_3= \!\!\! & \det \left[\begin{array}{cccc}
y&Y_1(x,y)\\
y&Y_2(x,y)
\end{array}\right] \\[3.0ex]
= \!\!\! & a_{21}x^2y^2+a_{03}y^4+a_{2n+2,1}x^{2n+2}y^2+a_{2n,3}x^{2n}y^4\\
&+\cdots+a_{2,2n+1}x^2y^{2n+2}+a_{0,2n+3}y^{2n+4}.
\end{array}
$$
Similarly, we take the $y$ and $x^2$ as infinitesimal equivalence  in the
neighborhood of the origin in order
 to study the dynamical behavior of \eqref{eq4.11}
around the origin. So, $J_2$ becomes
\begin{equation}\begin{split}
J_3=&x^4(a_{21}x^2+a_{03}x^4+a_{2n+2,1}x^{2n+2}+a_{2n,3}x^{2n+2}\\
&+\cdots+a_{2,2n+1}x^{4n+2}+a_{0,2n+3}x^{4n+4}),
\end{split}\end{equation}
implying that $a_{21},\,a_{03},\,a_{2n+2,1},\,a_{2n,3},\,\cdots,\,
a_{2,2n+1},\, a_{0,2n+3}$ could be considered as the focal values of the
system. Therefore, for $n\geq2$, the origin of system \eqref{eq4.11} is at most a
$(n\!+\!4)$th-order weak focus. According to Theorem 2.2 in \cite{Liu-2010},
within a small neighborhood of the origin,
one can perturb the coefficients of system \eqref{eq4.11} to obtain
$n+4$ small-amplitude limit cycles around the elementary node $O(0, 0)$.
\end{proof}

Moreover, we have a similar theorem for this case.
\begin{theorem}\label{T4.4}
For $n\geq2$, the origin of system \eqref{eq4.11} is an analytic center
if and only if
\begin{equation}
a_{21}=a_{03}=a_{2n+2,1}=a_{2n,3}=\cdots=a_{2,2n+1}=a_{0,2n+3}=0.
\end{equation}
\end{theorem}

\begin{proof}
When $a_{21}=a_{03}=a_{2n+2,1}=a_{2n,3}=\cdots=a_{2,2n+1}=a_{0,2n+3}=0,$ system \eqref{eq4.11} can be rewritten as
\begin{equation}\label{eq4.16}
\begin{array}{l}
\begin{split}
\dot{x}=& \ y,\\
\dot{y}=& -2x^3+a_{12}xy^2+a_{2n+3,0}x^{2n+3}+a_{2n+1,2}x^{2n+1}y^2+\cdots+a_{1,2n+2}xy^{2n+2},\\
\end{split}
\end{array}
\end{equation}
Obviously, system \eqref{eq4.16} is symmetric with the $y$-axis.
According to Theorem 11 in \cite{Li-2013a}, the origin is an analytic
center of system \eqref{eq4.11}.
\end{proof}

\section{Conclusion}
\noindent In this paper, two classes of lopsided systems have been studied
on their analytic integrable conditions and bifurcation of limit cycles.
We have obtained some analytic integrability conditions for each class of
the systems for case $n=1$.
By using certain transformations or integrating factors,
we have proved that all conditions are sufficient and necessary.
For case $n \ge 2$, we have constructed different comparison systems for each class of
the systems and shown that $n+4$ limit cycles may bifurcate
from the origin of each system.
In addition, conditions for the origin being an analytic center
are obtained simultaneously.

\section{Appendix}

Detailed recursive MATHEMATICA code to compute the quasi-Lyapunov constants at the origin of system (13):
c [0,0]=0, c [1, 0]=0, c [0, 1]=0, c [2, 0]=0, c [1, 1]=0, c [0, 2]=1; when k<0 or j<0, c [k,j]=0; else
\begin{equation*}
\begin{split}
c[k,j]&=-\frac{1}{k (1 + s)}(-10 a_{50} c[-5 + k, 1 + j] + a_{50} k c[-5 + k, 1 + j] -5 a_{50} s c[-5 + k, 1 + j]\\
& + a_{50} k s c[-5 + k, 1 + j] -8 a_{41} c[-4 + k, j] + a_{41} k c[-4 + k, j] - 4 a_{41} s c[-4 + k, j]\\
 &+a_{41} k s c[-4 + k, j] - 4 c[-4 + k, 2 + j] - 2 j c[-4 + k, 2 + j] -4 s c[-4 + k, 2 + j]\\
 & - 2 j s c[-4 + k, 2 + j] -6 a_{32} c[-3 + k, -1 + j] + a_{32} k c[-3 + k, -1 + j]\\
 & -3 a_{32} s c[-3 + k, -1 + j]+ a_{32} k s c[-3 + k, -1 + j]+ a_{30} k s c[-3 + k, 1 + j]\\
 &  -6 a_{30} c[-3 + k, 1 + j] + a_{30} k c[-3 + k, 1 + j] -3 a_{30} s c[-3 + k, 1 + j]\\
 &  -4 a_{23} c[-2 + k, -2 + j] + a_{23} k c[-2 + k, -2 + j] -2 a_{23} s c[-2 + k, -2 + j]\\
 & + a_{23} k s c[-2 + k, -2 + j] -4 a_{21} c[-2 + k, j] + a_{21} k c[-2 + k, j] - 2 a_{21} s c[-2 + k, j]\\
 & +a_{21} k s c[-2 + k, j] - 2 a_{14} c[-1 + k, -3 + j] +a_{14} k c[-1 + k, -3 + j]\\
 & - a_{14} s c[-1 + k, -3 + j]+a_{14} k s c[-1 + k, -3 + j] - 2 a_{12} c[-1 + k, -1 + j]\\
 &  +a_{12} k c[-1 + k, -1 + j] - a_{12} s c[-1 + k, -1 + j]+a_{12} k s c[-1 + k, -1 + j] \\
 & + a_{05} k c[k, -4 + j] +a_{05} k s c[k, -4 + j] + a_{03} k c[k, -2 + j] + a_{03} k s c[k, -2 + j]).\\
 \end{split}
\end{equation*}
\begin{equation*}
\begin{split}
w[m]&=9 a_{50} c[-4 + m, 0] - a_{50} m c[-4 + m, 0] + 4 a_{50} s c[-4 + m, 0] -a_{50} m s c[-4 + m, 0]\\
& + 7 a_{41} c[-3 + m, -1]- a_{41} m c[-3 + m, -1] +3 a_{41} s c[-3 + m, -1]\\
 & - a_{41} m s c[-3 + m, -1] + 2 c[-3 + m, 1] +2 s c[-3 + m, 1]+ 5 a_{32} c[-2 + m, -2]\\
 & - a_{32} m c[-2 + m, -2] +2 a_{32} s c[-2 + m, -2] - a_{32} m s c[-2 + m, -2] \\
 &+ 5 a_{30} c[-2 + m, 0] -a_{30} m c[-2 + m, 0] + 2 a_{30} s c[-2 + m, 0] - a_{30} m s c[-2 + m, 0] \\
 &+3 a_{23} c[-1 + m, -3] - a_{23} m c[-1 + m, -3] + a_{23} s c[-1 + m, -3] -a_{23} m s c[-1 + m, -3] \\
 &+ 3 a_{21} c[-1 + m, -1] - a_{21} m c[-1 + m, -1] +a_{21} s c[-1 + m, -1] - a_{21} m s c[-1 + m, -1] \\
 &+ a_{14} c[m, -4] -a_{14} m c[m, -4] - a_{14} m s c[m, -4] + a_{12} c[m, -2] - a_{12} m c[m, -2] \\
 &-a_{12} m s c[m, -2] - a_{05} c[1 + m, -5] - a_{05} m c[1 + m, -5] -a_{05} s c[1 + m, -5]\\
 & - a_{05} m s c[1 + m, -5]- a_{03} c[1 + m, -3] -a_{03} m c[1 + m, -3] - a_{03} s c[1 + m, -3]\\
 & - a_{03} m s c[1 + m, -3] -c[1 + m, -1] - m c[1 + m, -1] - s c[1 + m, -1] - m s c[1 + m, -1].
\end{split}
\end{equation*}

Detailed recursive MATHEMATICA code to compute the quasi-Lyapunov constants at the origin of system (24):
c [0,0]=0, c [1, 0]=0, c [0, 1]=0, c [2, 0]=0, c [1, 1]=0, c [0, 2]=1; when k<0 or j<0, c [k,j]=0; else
\begin{equation*}
\begin{split}
c[k,j]&=-\frac{1}{k (1 + s)} (2 a_{50} c[-6 + k, 2 + j] + a_{50} j c[-6 + k, 2 + j] +2 a_{50} s c[-6 + k, 2 + j]\\
& + a_{50} j s c[-6 + k, 2 + j]+a_{41} j c[-5 + k, 1 + j] + a_{41} s c[-5 + k, 1 + j] \\
&+ a_{41} j s c[-5 + k, 1 + j]- 2 a_{32} c[-4 + k, j] +a_{32} j c[-4 + k, j]+ a_{32} j s c[-4 + k, j]\\
&  - 4 c[-4 + k, 2 + j] -2 j c[-4 + k, 2 + j]- 4 s c[-4 + k, 2 + j]-2 j s c[-4 + k, 2 + j]\\
&  - 4 a_{23} c[-3 + k, -1 + j]+a_{23} j c[-3 + k, -1 + j]- a_{23} s c[-3 + k, -1 + j] \\
&  +a_{23} j s c[-3 + k, -1 + j] + a_{21} j c[-3 + k, 1 + j] +a_{21} s c[-3 + k, 1 + j]\\
&+ a_{21} j s c[-3 + k, 1 + j] -6 a_{14} c[-2 + k, -2 + j] + a_{14} j c[-2 + k, -2 + j] \\
&-2 a_{14} s c[-2 + k, -2 + j]+ a_{14} j s c[-2 + k, -2 + j] -2 a_{12} c[-2 + k, j] + a_{12} j c[-2 + k, j] \\
&+ a_{12} j s c[-2 + k, j]-8 a_{05} c[-1 + k, -3 + j] + a_{05} j c[-1 + k, -3 + j]\\
&  -3 a_{05} s c[-1 + k, -3 + j]+ a_{05} j s c[-1 + k, -3 + j]-4 a_{03} c[-1 + k, -1 + j]\\
&  + a_{03} j c[-1 + k, -1 + j]-a_{03} s c[-1 + k, -1 + j] + a_{03} j s c[-1 + k, -1 + j]).\\
\end{split}\end{equation*}
\begin{equation*}
\begin{split}
w[m]&=-a_{50} c[-5 + m, 1] - a_{50} s c[-5 + m, 1] + a_{41} c[-4 + m, 0] +3 a_{32} c[-3 + m, -1]\\
& + a_{32} s c[-3 + m, -1]+ 2 c[-3 + m, 1] +2 s c[-3 + m, 1] + 5 a_{23} c[-2 + m, -2] \\
&+ 2 a_{23} s c[-2 + m, -2] +a_{21} c[-2 + m, 0]+ 7 a_{14} c[-1 + m, -3] + 3 a_{14} s c[-1 + m, -3]\\
&+3 a_{12} c[-1 + m, -1] + a_{12} s c[-1 + m, -1] + 9 a_{05} c[m, -4]+4 a_{05} s c[m, -4] \\
& + 5 a_{03} c[m, -2] + 2 a_{03} s c[m, -2] - c[1 + m, -1] -m c[1 + m, -1]\\
& - s c[1 + m, -1] - m s c[1 + m, -1].
\end{split}
\end{equation*}

\end{document}